\documentclass[12pt]{amsart}
\usepackage{amssymb}
\usepackage{graphicx}
\usepackage[cmtip,all]{xy}
\usepackage[margin=1.0in]{geometry}
\newtheorem{theorem}{Theorem}[section]
\newtheorem{lemma}[theorem]{Lemma}
\newtheorem{proposition}[theorem]{Proposition}
\newtheorem{corollary}[theorem]{Corollary}
\theoremstyle{definition}
\newtheorem{definition}[theorem]{Definition}

\theoremstyle{remark}
\newtheorem{remark}[theorem]{Remark}
\numberwithin{equation}{section}

\newcommand{\quotient}{/\hspace{-1.2mm}/}
\newcommand{\ChowQ}{/\hspace{-1.2mm}/_{Ch}}
\newcommand{\M}{\overline{\mathcal{M}}}

\begin{document}

\title{Conformal Blocks and Rational Normal Curves}
\author{Noah Giansiracusa}
\address{Department of Mathematics, Brown University, Providence RI 02912}
\email{noahgian@math.brown.edu}
\thanks{Partial support provided by funds from NSF DMS-0901278.}
\commby{Michel Brion}

\begin{abstract}
We prove that the Chow quotient parameterizing configurations of $n$ points in $\mathbb{P}^d$ which generically lie on a rational normal curve is isomorphic to $\M_{0,n}$, generalizing the well-known $d=1$ result of Kapranov.  In particular, $\M_{0,n}$ admits birational morphisms to all the corresponding geometric invariant theory (GIT) quotients.  For symmetric linearizations the polarization on each GIT quotient pulls back to a divisor that spans the same extremal ray in the symmetric nef cone of $\M_{0,n}$ as a conformal blocks line bundle.  A symmetry in conformal blocks implies a duality of point-configurations that comes from Gale duality and generalizes a result of Goppa in algebraic coding theory.  In a suitable sense, $\M_{0,2m}$ is fixed pointwise by the Gale transform when $d=m-1$ so stable curves correspond to self-associated configurations.
\end{abstract}

\maketitle

\section{Introduction}

\subsection{Configurations of points} A natural way to compactify the moduli space of $n$ distinct points on the line, $\mathcal{M}_{0,n} = ((\mathbb{P}^1)^n\smallsetminus \{\text{diagonals}\})/\text{Aut}(\mathbb{P}^1)$, is to allow the points to collide by reintroducing the diagonals.  One can take a geometric invariant theory (GIT) quotient $(\mathbb{P}^1)^n\quotient_L\text{SL}_2$, but this requires a choice of linearization $L$.  Kapranov proved that the Chow quotient $(\mathbb{P}^1)^n\ChowQ\text{SL}_2$ is isomorphic to the standard compactification $\M_{0,n}$ \cite[Theorem 4.1.8]{Kap93} and that in general the Chow quotient maps to all GIT quotients \cite[Theorem 0.4.3]{Kap93}: $\M_{0,n} \cong (\mathbb{P}^1)^n\ChowQ\text{SL}_2 \rightarrow (\mathbb{P}^1)^n\quotient_L\text{SL}_2.$  

We add more flexibility by considering configurations of points in $\mathbb{P}^d$.  If \[U_{d,n} := \{(p_1,\ldots,p_n)\in(\mathbb{P}^d)^n~|~p_i\text{ are distinct points of a rational normal curve}\}\] then $U_{d,n}/\text{SL}_{d+1} \cong \mathcal{M}_{0,n}$ for $d \le n-3$.  The key is that now we compactify not only by allowing the points to collide, but also by degenerating the rational normal curve supporting the points.  By taking the topological closure \[V_{d,n} := \overline{U}_{d,n} \subseteq (\mathbb{P}^d)^n\] with reduced induced structure, we get the space of configurations that lie on what we call a \emph{quasi-Veronese} curve (see Definitions \ref{def:Veronese}, \ref{def:locus}, and Lemma \ref{lem:Vdef}).  The GIT quotients $V_{d,n}\quotient_L\text{SL}_{d+1}$ provide a family of birational models of $\M_{0,n}$.  For $d=1$ these have been studied extensively, e.g., \cite{DM86,GHP88,Kap93,GIT94,KM96a,Hu99,AL02,Has03,Pol95,AS08,Bol10,HMSV10}. For $d=2$ they occur as a special case of a construction in \cite{Sim08,GS09}. For $d \ge 3$ they do not seem to have appeared in the literature previously.

\begin{theorem}\label{thm:morphism}
For any $d \le n-3$ there is an isomorphism $V_{d,n}\ChowQ\text{SL}_{d+1}\cong\M_{0,n}$, and for any effective linearization $L$ on $V_{d,n}$ there is a birational morphism $\varphi : \M_{0,n} \rightarrow V_{d,n}\quotient_L\text{SL}_{d+1}$ extending $\mathcal{M}_{0,n}~\widetilde{\rightarrow}~U_{d,n}/\text{SL}_{d+1}$.\end{theorem}

We prove this by adapting ideas from \cite{GG11} to show that $\M_{0,n}$ admits a map to $\text{Chow}((\mathbb{P}^d)^n)$ that is an isomorphism onto its image, $V_{d,n}\ChowQ \text{SL}_{d+1}$.  The statement about GIT quotients then follows from Kapranov's general result on Chow and GIT quotients.  We provide an explicit description of $\varphi$ in \S\ref{section:modular}, compute which F-curves it contracts in \S\ref{section:Fcurves}, and show that it factors through Hassett's space of weighted pointed curves $\M_{0,L}$ in Proposition \ref{prop:factor}.

\subsection{Conformal blocks} Conformal field theory, first studied in physics then adopted by the mathematical community, has found applications in many areas.  See \cite{TUY89} for background.  For our purposes, it leads to a collection of vector bundles on $\M_{g,n}$.  The data required for each one is a Lie algebra, a non-negative integer called the level, and an $n$-tuple of dominant integral weights.  For $g=0$ these bundles are globally generated \cite[Lemma 2.5]{Fak11}.  For $\mathfrak{sl}_2$ with the so-called critical level, the weights can be chosen so that the determinant line bundle of the conformal blocks vector bundle spans the same ray in $N^1(\M_{0,n})$ as the pull-back of the GIT polarization on $(\mathbb{P}^1)^n\quotient_L\text{SL}_2$ along Kapranov's morphism $\M_{0,n} \rightarrow (\mathbb{P}^1)^n\quotient_L\text{SL}_2$ \cite[Theorem 4.5]{Fak11}.

In \cite{AGSS10} a specific family of conformal blocks line bundles is studied, corresponding to $\mathfrak{sl}_n$ with level $l=1$ and fundamental dominant weights $(\omega_k,\ldots,\omega_k), k\in\{2,\ldots,n-2\}$.  Let us denote these by $D_k^{\mathfrak{sl}_n}$.  The spaces $V_{d,n}$ provide the correct generalization of $V_{1,n} = (\mathbb{P}^1)^n$ to extend Fakhruddin's result about $\mathfrak{sl}_2$ bundles to these $\mathfrak{sl}_n$ bundles, in the following sense.  Each GIT quotient $V_{d,n}\quotient_L\text{SL}_{d+1}$ comes with a polarization, which by Theorem \ref{thm:morphism} can be pulled back to a line bundle on $\M_{0,n}$.  For the $S_n$-invariant linearization $L$ denote the GIT polarization by $\mathcal{L}_d\in\text{Pic}(V_{d,n}\quotient_L\text{SL}_{d+1})$.

\begin{theorem}\label{thm:CB}
The line bundles $\varphi^*\mathcal{L}_d$ and $D^{\mathfrak{sl}_n}_{d+1}$ span the same ray in $\text{N}^1(\M_{0,n})$.
\end{theorem}

By \cite[Remark 5.3]{Fak11}, the line bundle $D^{\mathfrak{sl}_n}_2$ can be identified with the $\mathfrak{sl}_2$ conformal blocks determinant inducing the map $\M_{0,n} \rightarrow (\mathbb{P}^1)^n\quotient_L \text{SL}_2$, so the $d=1$ case of this theorem is due to Fakhruddin.

\begin{corollary}\label{cor:extremal}
For $d=1,\ldots,\lfloor\frac{n}{2}\rfloor-1$, the line bundles $\varphi^*\mathcal{L}_d$ span distinct extremal rays of the symmetric nef cone of $\M_{0,n}$.
\end{corollary}

In \cite{AGSS10} it was shown that $D^{\mathfrak{sl}_n}_2,\ldots,D^{\mathfrak{sl}_n}_{\lfloor\frac{n}{2}\rfloor}$ span distinct extremal rays of the symmetric nef cone.  They prove this by finding, for each $d$, a family of $\rho -1$ independent curves contracted by the complete linear system $|D^{\mathfrak{sl}_n}_{d+1}|$, where $\rho := \text{dim}(N^1(\M_{0,n})^{S_n})$.  Because of this extremality result, to prove Theorem \ref{thm:CB} it is enough to show that the morphism $\varphi$ contracts the same set of independent curves.  We do this in \S\ref{section:GIT} by studying a formula of Fakhruddin for the degrees of conformal blocks bundles restricted to F-curves.

\subsection{Gale duality} The symmetry of the Dynkin diagram $\circ-\circ-\cdots-\circ-\circ$ for $\mathfrak{sl}_n$ implies that $D^{\mathfrak{sl}_n}_{k} = D^{\mathfrak{sl}_n}_{n-k}$.  This observation, which was brought to my attention by Valery Alexeev, together with Theorem \ref{thm:CB}, implies the following:

\begin{corollary}\label{cor:Gale}
For $S_n$-invariant $L$ there is an isomorphism of normalizations $(V_{d,n}\quotient_L\text{SL}_{d+1})^\nu \cong (V_{n-d-2,n}\quotient_L\text{SL}_{n-d-1})^\nu$.
\end{corollary}

This result, as was suggested by Brendan Hassett, is related to a classical duality of configurations of points in projective space known as the Gale transform \cite[Corollary III.1]{DO88}: \[(\mathbb{P}^d)^n\quotient_L\text{SL}_{d+1} \cong (\mathbb{P}^{n-d-2})^n\quotient_L\text{SL}_{n-d-1}.\]  The Gale transform of a configuration of distinct points lying on a rational normal curve in $\mathbb{P}^d$ lies on a rational normal curve in $\mathbb{P}^{n-d-2}$.  This was proven by Goppa in the context of algebraic coding theory to show that the dual of a Goppa code is a Goppa code.  See \cite{EP00} for a modern treatment, generalization, and reference to this historical anecdote.  In \S\ref{section:configs} we explain how Goppa's result, when combined with the Dolgachev-Ortland form of Gale duality, implies an isomorphism \[V_{d,n}\quotient_L\text{SL}_{d+1} \cong V_{n-d-2,n}\quotient_L\text{SL}_{n-d-1}\] which we interpret as saying that the Gale transform preserves configurations supported on quasi-Veronese curves, not just rational normal curves.  It is curious that, up to normalization, this also follows from a seemingly unrelated symmetry of conformal blocks, via Theorem \ref{thm:CB}.

For $n=2m$ and $d=m-1$ one can discuss self-associated configurations, i.e., points of $(\mathbb{P}^{m-1})^{2m}\ChowQ\text{SL}_m$ which are fixed by the Gale transform.  By Theorem \ref{thm:morphism} we can identify $\M_{0,2m}$ with $V_{m-1,2m}\ChowQ\text{SL}_m$ and thus view it as a subvariety of $(\mathbb{P}^{m-1})^{2m}\ChowQ\text{SL}_m$.  In \S\ref{section:selfassoc} we show that each stable curve $(C,p_1,\ldots,p_{2m})$ corresponds under this identification to a self-associated configuration.

\subsection*{Acknowledgements.}
I would like to thank Dan Abramovich, Valery Alexeev, Najmuddin Fakhruddin, Angela Gibney, Danny Gillam, Brendan Hassett, and Dave Swinarski for helpful conversations, and especially Gibney for suggesting this project.  

\section{Quasi-Veronese curves}

Throughout we work over $\mathbb{C}$.  We are interested in configurations of points lying on rational normal curves and their degenerations.

\begin{definition}\label{def:Veronese}
Let $\mathcal{H}_d \subseteq \text{Hilb}^{dt+1}(\mathbb{P}^d)$ denote the closed component of the Hilbert scheme parameterizing rational normal curves of degree $d$ and their degenerations.  We call such curves \emph{quasi-Veronese} curves.
\end{definition}

\begin{definition}\label{def:locus}
Denote by $\text{RNC}(d,n)\subseteq\mathcal{H}_d\times(\mathbb{P}^d)^n$ the incidence locus of $n$-pointed quasi-Veronese curves, and let $V_{d,n} := \pi_2(\text{RNC}(d,n))\subseteq (\mathbb{P}^d)^n$, where $\pi_2 : \mathcal{H}_d\times(\mathbb{P}^d)^n \rightarrow (\mathbb{P}^d)^n$ is the projection.
\end{definition}

Recall from the introduction that $U_{d,n}\subseteq(\mathbb{P}^d)^n$ denotes the locus of configurations of $n$ distinct points that lie on a rational normal curve in $\mathbb{P}^d$.

\begin{lemma}\label{lem:Vdef}
The locus $V_{d,n}$ coincides with the closure $\overline{U}_{d,n}\subseteq(\mathbb{P}^d)^n$.
\end{lemma}

\begin{proof} Since $\text{RNC}(d,n)\subseteq\mathcal{H}_d\times(\mathbb{P}^d)^n$ is closed and $\pi_2$ is proper we have that $V_{d,n}$ is closed, and clearly $U_{d,n} \subseteq V_{d,n}$, so $\overline{U}_{d,n}\subseteq V_{d,n}$.  On the other hand, the containment $V_{d,n}\subseteq \overline{U}_{d,n}$ holds since any quasi-Veronese curve is the flat limit of rational normal curves and hence a configuration of points on the former must lie in the closure of configurations on the latter. 
\end{proof}

There is a natural action of $\text{SL}_{d+1}$ on $\text{RNC}(d,n)$ and on $V_{d,n}$.  Because $\mathcal{H}_1 = \{\text{pt}\}$ we have $\text{RNC}(1,n) = V_{1,n} = (\mathbb{P}^1)^n$, and in the introduction there is an extensive bibliography of papers that deal with this space and its various quotients by $\text{SL}_2$.  In \cite{GS09} it was proven that $\M_{0,n}$ admits a birational morphism to each GIT quotient $\text{RNC}(2,n)\quotient_L\text{SL}_3$.  For linearizations $L$ that are trivial on $\mathcal{H}_2 = \mathbb{P}^5$, the GIT quotient is obtained by first applying the induced morphism $|L| : \text{RNC}(2,n) \rightarrow V_{2,n}$ which ``forgets'' the underlying conic and then taking the usual GIT quotient of the resulting space. Thus a corollary is the existence of the morphisms $\varphi$ in Theorem \ref{thm:morphism} when $d=2$.  In this paper we only study quotients of $V_{d,n}$, but it would be interesting to know if Theorem \ref{thm:morphism} generalizes from these quotients to quotients of $\text{RNC}(d,n)$ as it does in the cases $d=1$ (trivially) and $d=2$ (by \cite{GS09}).

\begin{remark}
Since this paper concerns the locus $V_{d,n}$ where only points are parameterized, not the curves supporting them, we could have used different compactifications of the space of rational normal curves.  For instance, one could take the Kontsevich stable map space $\M_{0,n}(\mathbb{P}^d,d)$ and use the product of evaluation maps $\M_{0,n}(\mathbb{P}^d,d) \rightarrow \mathbb{P}^d$ to define the same locus $V_{d,n}\subseteq (\mathbb{P}^d)^n$.
\end{remark}

\section{Chow Quotients}

\subsection{Background} We briefly recall the definition of a Chow quotient \cite[\S0.1]{Kap93} (see also \cite[\S3]{Hu05a}).   If an algebraic group $G$ acts on a projective variety $X$, then there exists a Zariski dense $G$-invariant subset $U\subseteq X$ for which all the orbit closures $\overline{Gu} \subseteq X$, $u\in U$, have the same dimension $r$ and homology class $\delta\in H_{2r}(X,\mathbb{Z})$.  This induces an embedding $U/G \hookrightarrow \text{Chow}(X,\delta)$ into the Chow variety parameterizing cycles in $X$ with homology class $\delta$.  By definition, the Chow quotient is the closure of the image: \[X\ChowQ G := \overline{U/G}\subseteq\text{Chow}(X,\delta).\]  If $G$ is reductive, then there is a birational morphism from the Chow quotient to each GIT quotient $X\quotient_L G$ for which there is a stable point, i.e., such that the linearization $L$ is effective \cite[Theorem 0.4.3]{Kap93}.

\subsection{Proof of Theorem \ref{thm:morphism}}\label{section:chow} Fix $n \ge 4$ and $1\le d \le n-3$.  We first prove the existence of a morphism $\M_{0,n} \rightarrow V_{d,n}\ChowQ\text{SL}_{d+1}$ extending $\mathcal{M}_{0,n}~\widetilde{\rightarrow}~U_{d,n}/\text{SL}_{d+1}$ and then show that it is an isomorphism.  The statement about GIT quotients follows from the preceding general remark.

\subsubsection*{Setup} The inclusion $V_{d,n}\subseteq (\mathbb{P}^d)^n$ induces a closed embedding \[V_{d,n}\ChowQ\text{SL}_{d+1} \hookrightarrow (\mathbb{P}^d)^n\ChowQ\text{SL}_{d+1}.\]  These latter Chow quotients were introduced by Kapranov.  In \cite[Proposition 2.1.7]{Kap93} it was shown that the locus $U$ of generic points in the definition of the Chow quotient may be taken to be those configurations such that any $m$ points, $m\le d+1$, span a $\mathbb{P}^{m-1}$.  It is a classical fact that distinct points on a rational normal curve satisfy this property, so we may take the generic locus to be $U_{d,n}$ and view $V_{d,n}\ChowQ\text{SL}_{d+1}$ as the closure of $U_{d,n}/\text{SL}_{d+1}$ inside $\text{Chow}((\mathbb{P}^d)^n)$.  Our goal then is to construct a morphism $\M_{0,n} \rightarrow \text{Chow}((\mathbb{P}^d)^n)$ extending the embedding $\mathcal{M}_{0,n} = U_{d,n}/\text{SL}_{d+1} \hookrightarrow \text{Chow}((\mathbb{P}^d)^n)$.  The image of such a map is necessarily contained in $V_{d,n}\ChowQ\text{SL}_{d+1}$.  Moreover, there can be at most one such extension, since the Chow variety is separated.

\subsubsection*{Extending the map} Let $\Gamma\subseteq \mathcal{M}_{0,n}\times\text{Chow}((\mathbb{P}^d)^n)$ be the graph of the morphism $\mathcal{M}_{0,n} \hookrightarrow \text{Chow}((\mathbb{P}^d)^n)$ and $\overline{\Gamma}$ its closure in $\M_{0,n}\times \text{Chow}((\mathbb{P}^d)^n)$, with reduced induced structure.  It suffices to show that $\pi_1 : \overline{\Gamma} \rightarrow \M_{0,n}$ is an isomorphism, for then $\pi_2\pi_1^{-1} : \M_{0,n} \rightarrow \text{Chow}((\mathbb{P}^d)^n)$ is the desired extension.  To show that $\pi_1$ is an isomorphism, we use that it is birational and $\M_{0,n}$ is normal to reduce to proving it is finite.  And for this, we use the fact that $\overline{\Gamma}$ is proper to reduce to showing that $\pi_1$ is quasi-finite.  In fact, it is enough to show that the fiber over each closed point is finite.  Indeed, each fiber has finite cardinality if and only if it is zero-dimensional, and by upper-semicontinuity the set of points with positive-dimensional fibers is closed, so if this set were nonempty then it would contain a closed point.  We now proceed to show that the fiber over each closed point is finite---in fact, that it has cardinality one.

Let $x\in\M_{0,n}(\mathbb{C})$ and $(x,y_i)\in\overline{\Gamma}_x$, $i=1,2$.  We want to show $y_1=y_2$.  There are points $(x'_i,y'_i)\in\Gamma$ specializing to $(x,y_i)$, and by \cite[III.7.1.9]{EGA60} there are DVRs $R_i$ and morphisms $\text{Spec }R_i \rightarrow \overline{\Gamma}$ sending the generic point $\text{Spec }K_i$ to $(x'_i,y'_i)$ and the closed point $\text{Spec }k_i$ to $(x,y_i)$.  Thus we are reduced to the following situation.  Given a stable pointed curve $(C_k,p_1,\ldots,p_n)\in\M_{0,n}$, write it as the special fiber in a family $C_R \rightarrow \text{Spec }R$ of stable pointed curves over a DVR such that the generic fiber $C_K \rightarrow \text{Spec }K$ is smooth.  There is an algebraic cycle in $(\mathbb{P}^d_K)^n$ obtained by applying the $d^{\text{th}}$ Veronese map to the marked points of $C_K \cong \mathbb{P}_K^1$ and taking the $\text{SL}_{d+1}$-orbit closure of the resulting configuration.  This cycle limits to a cycle in $(\mathbb{P}^d_k)^n$, and we must show that this limit cycle is independent of the smoothing of $C_k$.  We will do this by explicitly describing this limit cycle.

\subsubsection*{The limit cycle} For $\mathcal{C} := (C,p_1,\ldots,p_n)\in\M_{0,n}$, write $C=C_1\cup\cdots \cup C_r$ as a union of irreducible components, each isomorphic to $\mathbb{P}^1$.  For each partition $d = d_1 + \cdots + d_r$ with $d_i \ge 0$ choose auxiliary smooth points $q_1,\ldots,q_d\in C$ so that $C_i$ has $d_i$ points, and consider the line bundle $L := \mathcal{O}_C(q_1 + \cdots + q_d)$.  It has vanishing higher cohomology, since by Serre duality $h^1(L) = h^0(\omega_C\otimes L^{-1})=0$.  So by Riemann-Roch $h^0(L) = \chi(L) = d +1$.  Moreover, it is basepoint free, so it induces a morphism $\phi_L : C \rightarrow \mathbb{P}^d$ sending $p_1,\ldots,p_n$ to a configuration of $n$ not necessarily distinct points in $\mathbb{P}^d$: each $C_i$ gets sent to a rational normal curve of degree $d_i$ in the projective space that it spans, and if $d_i = 0$ then $C_i$ gets contracted and all the marked points on it have the same image in $\mathbb{P}^d$.  It follows from Lemma \ref{lem:lifting} below (by taking a smoothing of $C$ over a DVR) that the resulting curve is a quasi-Veronese curve, so this configuration lies in $V_{d,n}$.  For a fixed partition, the choice of auxiliary points leads to maps $C \rightarrow \mathbb{P}^d$ which differ only by automorphisms of $\mathbb{P}^d$, so the orbit closure \[Z_{d_1\cdots d_r} := \overline{\text{SL}_{d+1}(\phi_L(p_1),\ldots,\phi_L(p_n))}\subseteq (\mathbb{P}^d)^n\] depends only on the partition of $d$ and not on the line bundle $L$ itself.  Let \[Z(\mathcal{C}) := \bigcup_{d=d_1 + \cdots + d_r}Z_{d_1 \cdots d_r} \subseteq (\mathbb{P}^d)^n\] with reduced induced subscheme structure.  Not every orbit closure $Z_{d_1\cdots d_r}$ is full-dimensional, but it follows from Lemmas \ref{lem:DVR} and \ref{lem:ineq} below that $Z(\mathcal{C})$ is of pure dimension $(d+1)^2-1$ so that each $Z_{d_1\cdots d_r}$ must be contained in a full-dimensional orbit closure appearing in this union.  We claim $Z(\mathcal{C})\in\text{Chow}((\mathbb{P}^d)^n)$ is the limit cycle no matter how $\mathcal{C}$ is written as a limit of curves.  

\begin{lemma} \label{lem:lifting} Let $(R,\mathfrak{m})$ be a DVR with fraction field $K$ and residue field $k$, and let $(C_R,p_1,\ldots,p_n)\in\M_{0,n}(R)$.  Then any smooth points $q_1,\ldots,q_d : \text{Spec }k \rightarrow C_k$ extend to smooth points $q_1,\ldots,q_d : \text{Spec }R \rightarrow C_R$, and every section of the line bundle $L_k := \mathcal{O}_{C_k}(q_1+\cdots+q_d)$ lifts to a section of $L_R := \mathcal{O}_{C_R}(q_1+\cdots+q_d)$.
\end{lemma}

\begin{proof}
The statement about extending points follows from smoothness, and the fact that sections lift follows from Grauert's theorem \cite[Corollary III.12.9]{Har77} due to the vanishing of higher cohomology mentioned above.
\end{proof}

We can now show that $Z(\mathcal{C})$ is contained in the limit cycle---for any family, not just a smoothing.

\begin{lemma}\label{lem:DVR}
Let $R$ be a DVR as above, $\mathcal{C}_R\in\M_{0,n}(R)$, and write $\mathcal{C}_K,\mathcal{C}_k$ for the general and special fibers, respectively.  Then \[Z(\mathcal{C}_k) \subseteq \overline{Z(\mathcal{C}_K)}_k\] where the closure is taken in $(\mathbb{P}_R^d)^n$.
\end{lemma}

\begin{proof}
If $r$ denotes the number of components of $C_k$ then by the definition of $Z(\mathcal{C}_k)$ we must show the containment $Z_{d_1\cdots d_r} \subseteq \overline{Z(\mathcal{C}_K)}_k$ for each partition $d=d_1+\cdots + d_r$.  Since the right-hand side is closed and $\text{SL}_{d+1}$-invariant, it is enough to show that $\phi_{L_k}$ sends the marked points of $C_k$ into $\overline{Z(\mathcal{C}_K)}_k$, where as above $L_k = \mathcal{O}_{C_k}(q_1 + \cdots + q_d)$ is a line bundle of degree $d$ on $C_k$ determined by this partition and $\phi_{L_k}$ is the induced morphism.  By Lemma \ref{lem:lifting}, $L_k$ extends to a line bundle $L_R := \mathcal{O}_R(q_1 + \cdots + q_d)$ on $C_R$ and the sections inducing $\phi_{L_k} : C_k \rightarrow \mathbb{P}_k^d$ lift to give a map $\phi_{L_R} : C_R \rightarrow \mathbb{P}_R^d$.  The restriction of $\phi_{L_R}$ to $C_K$ is induced by $\mathcal{O}_{C_K}(q_1+\cdots+q_d)$ and sends the marked points of $C_K$ to a point of $Z(\mathcal{C}_K)$, so by continuity we are done. 
\end{proof}

\subsubsection*{Comparison of homology classes} If we show that the containment in Lemma \ref{lem:DVR} is an equality when $C_K$ is smooth, then we will have proven that the limit cycle is unique.  For this, it is enough to prove the following:

\begin{lemma}\label{lem:ineq} With notation as above and $C_K$ smooth, we have \[[Z(\mathcal{C}_k)] \geq [\overline{Z(\mathcal{C}_K)}_k]\in H_{2((d+1)^2-1)}((\mathbb{P}_k^d)^n,\mathbb{Z}).\]
\end{lemma}

\begin{proof}
Since $\overline{Z(\mathcal{C}_K)}_k$ is a specialization of a point in $\text{Chow}((\mathbb{P}^d)^n)$ corresponding to a generic orbit closure, it has the same homology class as a generic orbit closure.  So we must show $[Z(\mathcal{C})] \ge [Z(\mathcal{C'})]$, where $\mathcal{C}\in\M_{0,n}$ and $\mathcal{C}'\in\mathcal{M}_{0,n}$.  Lemma \ref{lem:DVR} implies that the homology class can only decrease when the curve degenerates, so we reduce to the case that $\mathcal{C}$ is maximally degenerate, i.e., that each component has exactly three special points.  Kapranov computed the homology class of a generic orbit closure, which we now recall.

By the K\"unneth formula, a basis for $H_{2((d+1)^2-1)}((\mathbb{P}^d)^n)$ is $[\mathbb{P}^{l_1}]\otimes\cdots\otimes[\mathbb{P}^{l_n}]$ where $0\le l_i\le d$ and $\sum_{i=1}^n l_i = (d+1)^2-1$.  By \cite[Proposition 2.1.7]{Kap93}, \[[Z(\mathcal{C}')]=\sum_{l_1 + \cdots + l_n = (d+1)^2-1}[\mathbb{P}^{l_1}]\otimes\cdots\otimes[\mathbb{P}^{l_n}]\] so the generic orbit closure has coefficient 1 at each basis element.  In other words, if $L_i\subseteq\mathbb{P}^d$ are generic linear subspaces of codimension $l_i$ then the intersection number of $L_1\times\cdots\times L_n$ with $Z(\mathcal{C}')$ is 1.  Our goal is to show that the intersection with $Z(\mathcal{C})$ is $\ge 1$.  

The assumption that $\mathcal{C}=(C,p_1,\ldots,p_n)$ is maximally degenerate implies that it has $n-2$ components and that the only partitions $d=d_1+\cdots+d_{n-2}$ that yield a full-dimensional cycle $Z_{d_1\cdots d_{n-2}}$ are those with $d_i \le 1$.  For each choice of integers $0\le l_i \le d$ with $\sum_{i=1}^n l_i = (d+1)^2 -1$, fix a generic product of linear subspaces $L_1 \times \cdots \times L_n$ as above.  By the construction of $Z(\mathcal{C})$, the proof will be complete if we show that there is a degree $d$ map $\psi : C \rightarrow \mathbb{P}^d$, sending each component to either a line or a point, such that $\psi(p_i)\in L_i$ for $i=1,\ldots,n$.  This follows from the following slightly more general result. 
\end{proof}

\begin{lemma} Let $\mathcal{C}=(C,p_1,\ldots,p_n)\in\M_{0,n}$ be a maximally degenerate curve.  For any integer $e\le d$ and generic linear subspaces $L_i\subseteq\mathbb{P}^d$ such that \[\sum_{i=1}^n\mathrm{codim}(L_i) = (d+1)(e+1)-1\] there is a unique degree $e$ map $\psi : C \rightarrow \mathbb{P}^d$, linear on components, with $\psi(p_i)\in L_i$ for $i=1,\ldots,n$.
\end{lemma}

\begin{proof}
Use induction on $n$.  For the base case $n=3$ we have that $C\cong \mathbb{P}^1$, the $L_i$ are either points or $\mathbb{P}^1$, and we must show that the result holds for $e\le d=1$.  If $e=1$, then $\sum_{i=1}^3\mathrm{codim}(L_i)=3$, so each $L_i$ is a point and indeed there is a unique automorphism of $\mathbb{P}^1$ sending $p_i$ to $L_i$ for $i=1,2,3$.  If $e=0$, then $\sum_{i=1}^3\mathrm{codim}(L_i)=1$, so one of the $L_i$ is a point, say $L_1$, and the other two are lines.  Thus $\psi : C \rightarrow \mathbb{P}^1$ is the map defined by $\psi(C) = L_1 \in \mathbb{P}^1$.

For $n >3$, consider a component $D\subseteq C$ with one node.  By the maximal degeneration hypothesis there are two points on $D$, say $p_1$ and $p_2$.  If $L_1\cap L_2 \ne \varnothing$, then contract $D$ to produce a stable, maximally degenerate curve $(C',p_2,\ldots,p_n)\in\M_{0,n-1}$.  By the inductive hypothesis, there is a unique map $\psi' : C' \rightarrow \mathbb{P}^d$ such that $(\psi'(p_2),\ldots,\psi'(p_n))\in (L_1\cap L_2) \times L_3 \times \cdots \times L_n$, since the $L_i$ are generic and hence $\text{codim}(L_1\cap L_2) = \text{codim}(L_1) + \text{codim}(L_2)$.  This extends uniquely to a map on $C$, with degree zero on $D$, satisfying the required properties.

So suppose now that $p_1,p_2\in D\subseteq C$ as before, but $L_1 \cap L_2 = \varnothing$.  Again consider the curve $(C',p_2,\ldots,p_n)\in\M_{0,n-1}$ obtained by forgetting $p_1$ and stabilizing.  By induction, there is a unique degree $e-1$ map $\psi' : C' \rightarrow \mathbb{P}^d$ such that $(\psi'(p_2),\ldots,\psi'(p_n))\in (L_1 + L_2) \times L_3 \times \cdots \times L_n$.  Indeed, \[\text{codim}(L_1+L_2) + \sum_{i=3}^n\text{codim}(L_i) = - d - 1 + \sum_{i=1}^n\text{codim}(L_i) = (d+1)((e-1)+1)-1.\] To complete the proof we must show that there is a unique way to extend $\psi'$ to a map $\psi$ on $C$, with degree 1 on $D$, such that $\psi(p_i)\in L_i$.  

If we label the node of $D$ by $q$, then $\psi'(p_i)\in L_i$ for $i=3,\ldots,n$, and $\psi'(q)\in L_1+L_2$.  All that remains is to show there is a unique line in $\mathbb{P}^d$ containing $\psi'(q)$ and intersecting both $L_1$ and $L_2$.  Consider the projective space $\mathbb{P}^m := L_1 + L_2 \subseteq\mathbb{P}^d$ and the projection $\mathbb{P}^m\setminus\{\psi'(q)\} \rightarrow \mathbb{P}^{m-1}$.  By assumption $L_1$ and $L_2$ are skew, so after projecting to $\mathbb{P}^{m-1}$ they intersect in a unique point, which means precisely that there is a unique line through $\psi'(q)$ intersecting both $L_1$ and $L_2$, as desired.
\end{proof}

\subsubsection*{The isomorphism} Since $\M_{0,n}$ is proper, its image in $\text{Chow}((\mathbb{P}^d)^n)$ is closed and so topologically coincides with $V_{d,n}\ChowQ\text{SL}_{d+1}$.  Thus all that remains in the proof of Theorem \ref{thm:morphism} is to show that the morphism $\M_{0,n} \rightarrow \text{Chow}((\mathbb{P}^d)^n)$ constructed above is an isomorphism onto its image.  For each $I\subseteq\{1,\ldots,n\}$ of size $d+3$ there is a projection $\pi_I : (\mathbb{P}^d)^n \rightarrow (\mathbb{P}^d)^{d+3}$.  Since $\pi_I$ is proper it induces a map $\text{Chow}((\mathbb{P}^d)^n) \rightarrow \text{Chow}((\mathbb{P}^d)^{d+3})$ obtained by pushing forward cycles \cite[Theorem 6.8]{Kol96}.  Since $\pi_I$ is $\text{SL}_{d+1}$-equivariant this restricts to a map $U_{d,n}/\text{SL}_{d+1} \rightarrow U_{d,d+3}/\text{SL}_{d+1}$ and hence a map on the topological closures $V_{d,n}/\text{SL}_{d+1} \rightarrow V_{d,d+3}/\text{SL}_{d+1}$ as well.  There are also stabilization morphisms $\M_{0,n} \rightarrow \M_{0,d+3}$, and the following diagram commutes:
\begin{equation}\label{comdiag}\xymatrix{\M_{0,n} \ar@{->>}[r] \ar[d] & V_{d,n}\ChowQ\text{SL}_{d+1} \ar@{^{(}->}[r] \ar[d] & \text{Chow}((\mathbb{P}^d)^n) \ar[d] \\ \prod_I \M_{0,d+3} \ar@{->>}[r] & \prod_I V_{d,d+3}\ChowQ\text{SL}_{d+1} \ar@{^{(}->}[r] & \prod_I\text{Chow}((\mathbb{P}^d)^{d+3})}
\end{equation}

\begin{lemma}\label{lem:assoc}
The map $\M_{0,d+3} \rightarrow V_{d,d+3}\ChowQ\text{SL}_{d+1}$ is an isomorphism.
\end{lemma}

\begin{proof}
By Gale duality (see \S\ref{section:Gale}) we have $(\mathbb{P}^d)^{d+3}\ChowQ\text{SL}_{d+1}\cong (\mathbb{P}^1)^{d+3}\ChowQ\text{SL}_{2}$ \cite[Corollary 2.3.14]{Kap93}, and $(\mathbb{P}^1)^{d+3}\ChowQ\text{SL}_{2} \cong \M_{0,d+3}$ by \cite[Theorem 4.1.8]{Kap93}, so there is an isomorphism $\M_{0,d+3}~\widetilde{\rightarrow}~(\mathbb{P}^d)^{d+3}\ChowQ\text{SL}_{d+1}$ which it is easy to see extends the embedding $\mathcal{M}_{0,d+3} \hookrightarrow \text{Chow}((\mathbb{P}^d)^{d+3})$.  By uniqueness of the extension this is the same map as in the statement of the lemma.
\end{proof}

\begin{lemma}\label{lem:isoimage}
For $4 \le k \le n$, the product of stabilization maps $\M_{0,n} \rightarrow \prod_I \M_{0,k}$ over $I\subseteq\{1,\ldots,n\}$ of size $k$ is an isomorphism onto its image.
\end{lemma}

\begin{proof}
There is a factorization $\M_{0,n} \rightarrow \prod_I \M_{0,k} \rightarrow \prod_J \M_{0,4}$ where the latter product is over all $J\subseteq\{1,\ldots,n\}$ of size 4, so the general result follows immediately from the case $k=4$.  The fact that the product of stabilization maps to $\M_{0,4}\cong\mathbb{P}^1$ is an isomorphism onto its image is proven in \cite[Theorem 1.3]{GG11}: it follows by induction on $n$ using the compatability of the stabilization morphisms with the boundary stratification of $\M_{0,n}$.
\end{proof}

Lemmas \ref{lem:isoimage} (with $k=d+3$) and \ref{lem:assoc} imply that $\M_{0,n} \rightarrow \prod_I V_{d,d+3}\ChowQ\text{SL}_{d+1}$ is an isomorphism onto its image, so by the commutativity of (\ref{comdiag}) the same is true of $\M_{0,n} \rightarrow V_{d,n}\ChowQ\text{SL}_{d+1}$.  This completes the proof.

\section{GIT quotients}\label{section:GIT}

Having proven the existence of the morphisms $\varphi$ in Theorem \ref{thm:morphism}, we now focus on deriving some of their properties.

\subsection{Explicit description} \label{section:modular}  Every cycle parameterized by $(\mathbb{P}^d)^n\ChowQ\text{SL}_{d+1}$ is a sum of closures of full-dimensional orbits with multiplicity one \cite[Corollary 2.2.6]{Kap93}.  By the inclusion $V_{d,n}\ChowQ\text{SL}_{d+1} \subseteq (\mathbb{P}^d)^n\ChowQ\text{SL}_{d+1}$ this holds for $V_{d,n}\ChowQ\text{SL}_{d+1}$ as well.  By \cite[Theorem 0.4.3]{Kap93} and its proof, for each linearization $L$ such that the semistable locus is nonempty, at least one of these orbit closures is semistable and all semistable ones are equivalent---in the sense that they determine the same point in the GIT quotient.  The Chow-GIT morphism $V_{d,n}\ChowQ\text{SL}_{d+1} \rightarrow V_{d,n}\quotient_L\text{SL}_{d+1}$ is defined by sending each sum of orbit closures to the corresponding semistable orbit class.  On the other hand, by the discussion of the limit cycle in our proof of Theorem \ref{thm:morphism}, the morphism $\M_{0,n} \rightarrow V_{d,n}\ChowQ\text{SL}_{d+1}$ sends $\mathcal{C} = (C,p_1,\ldots,p_n)\in\M_{0,n}$ to the sum of orbit closures $Z(\mathcal{C})$ obtained by mapping $C$ to a quasi-Veronese curve in all possible ways corresponding to partitions of $d$ among the components of $C$.  So for each linearization there is one partition that leads to a semistable configuration (or multiple partitions leading to GIT-equivalent configurations in the case of strictly semistable points) and $\varphi$ sends $\mathcal{C}$ to the orbit of this configuration.

\subsection{GIT Stability} \label{section:stability} Stability for the action of $\text{SL}_{d+1}$ on $V_{d,n}$ can be viewed through the inclusion $V_{d,n} \subseteq (\mathbb{P}^d)^n$ into the space of configurations of $n$ points in $\mathbb{P}^d$, where it is worked out in \cite[Example 3.3.24]{DH98}.  The $\text{SL}_{d+1}$-ample cone of fractional linearizations for $(\mathbb{P}^d)^n$ is $\mathbb{Q}_{>0}^n$ and we view a vector $L=(x_1,\ldots,x_n)\in\mathbb{Q}_{>0}^n$ as assigning a positive rational weight to each point.  A configuration is semistable if and only if the total weight lying in any linear subspace $W\subsetneq\mathbb{P}^d$ is at most $\frac{\text{dim}(W)+1}{d+1}\cdot \sum_{i=1}^n x_i$.  Multiplying $L$ by a positive constant does not affect stability, so one can assume $\sum_{i=1}^n x_i = d+1$.  In this case we have: \begin{center}\emph{A point of $V_{d,n}$ is semistable if, and only if, the weight $w$ in each linear subspace $W\subsetneq\mathbb{P}^d$ satisfies $w \le \mathrm{dim}(W) + 1$.}\end{center}

The semistable locus is then non-empty precisely when $\text{max}\{x_i\} \le 1$, so the space of effective linearizations can be identified with the hypersimplex \[\Delta(d+1,n) := \{(x_1,\ldots,x_n)\in\mathbb{Q}^n~|~0\le x_i \le 1, \sum_{i=1}^n x_i = d+1\}.\] In particular, the unique $S_n$-invariant linearization is $L=(\frac{d+1}{n},\ldots,\frac{d+1}{n})$.

This polytope of linearizations is subdivided into closed chambers such that on their interiors the corresponding GIT quotients are constant.  The walls for this decomposition are of the form $\sum_{i\in I} x_i=k$ for $I\subset\{1,\ldots,n\}$ and $1 \le k \le d$.  Thus for each $d\in\{1,\ldots,n-3\}$ and $L\in\Delta(d+1,n)$ there is a morphism $\varphi : \M_{0,n} \rightarrow V_{d,n}\quotient_L\text{SL}_{d+1}$ and these morphisms do not change as $L$ varies within a fixed chamber.

\subsection{F-curves}\label{section:Fcurves}  The boundary of $\M_{0,n}$ is stratified by the number of nodes on the curves it parameterizes.   The irreducible components of 1-strata are called F-curves.  These are isomorphic to $\M_{0,4}\cong\mathbb{P}^1$ and are obtained by attaching maximally degenerate chains of $\mathbb{P}^1$s to four points on a $\mathbb{P}^1$ (the empty chain consisting of a single point is allowed).  Varying the cross-ratio of the attaching points traces out the F-curve.  We call the chains \emph{legs} of the F-curve and the component with the attaching points the \emph{spine}.  The numerical equivalence class of an F-curve is determined by the partition of $\{1,\ldots,n\}$ into four subsets indicating which marked points lie on which leg \cite[Lemma 4.3]{KM96b}.

By the F-curve \emph{corresponding} to the partition $\{1,\ldots,n\} = \sqcup_{j=1}^4 N_j$, we mean the numerical equivalence class of F-curves with points indexed by $N_j$ on the $j^{\text{th}}$ leg.  By a \emph{symmetric} F-curve, we mean the image of an F-curve in $\M_{0,n}/S_n$.  In this case only the number of marked points on each leg is relevant, so we write $n_j := |N_j|$ and speak of the symmetric F-curve corresponding to $n=n_1 + n_2 +n_3 + n_4$.  By abuse of language we also refer to an F-curve in $\M_{0,n}$ corresponding to $n=n_1 + n_2 +n_3 + n_4$, by which we mean any F-curve whose image in $\M_{0,n}/S_n$ corresponds to this partition.

\begin{proposition}\label{prop:cont}
Let $L=(x_1,\ldots,x_n)\in\Delta(d+1,n)$.  Any F-curve corresponding to $\sqcup_{j=1}^4 N_j$ such that $\sum_{i\in N_j} x_i \ge \alpha_j,~j=1,\ldots,4$, for some integers $\alpha_j \ge 0$ with $\sum_{j=1}^4 \alpha_j = d$, is contracted by $\varphi : \M_{0,n} \rightarrow V_{d,n}\quotient_L\text{SL}_{d+1}$.
\end{proposition}

\begin{proof}
By the description in \S \ref{section:modular}, $\varphi$ sends a marked curve $(C,p_1,\ldots,p_n)\in\M_{0,n}$ with irreducible decomposition $C=C_1\cup\cdots\cup C_r$ to a semistable configuration of points lying on a quasi-Veronese curve $V=V_1\cup\cdots\cup V_r$.  This corresponds to a partition $d=d_1 + \cdots + d_r$ where each $V_i$ is either a point ($d_i = 0$) or a rational normal curve of degree $d_i \ge 1$.  Recall from \S \ref{section:stability} that $L=(x_1,\ldots,x_n)$ assigns weight $x_i$ to the marked point $p_i$ and a configuration in $\mathbb{P}^d$ is semistable if and only if there is weight at most 1 at a point, 2 on a line, 3 on a plane, etc.  

It follows that when a connected set of components $\cup_{i\in I}C_i$, $I\subseteq\{1,\ldots,r\}$, carries marked points of total weight $>m$ then the corresponding curve $\cup_{i\in I}V_i$ must have degree $\ge m$, since otherwise it would be contained in a $\mathbb{P}^{m-1}$, violating semistability.  Since $V$ has degree $d$, the spine will be contracted if the degrees of the legs add up to $d$.  By the preceding observation, this occurs if the weight of marked points on the $j^{\text{th}}$ leg of the F-curve is $> \alpha_j$ for integers $\alpha_j \ge 0$ satisfying $\sum_{j=1}^4 \alpha_j = d$.  In fact, it is enough to require $\sum_{i\in N_j}x_i \ge \alpha_j$, because if the weight on a leg is exactly $\alpha_j$, then the configuration is strictly semistable and the GIT quotient identifies it with a configuration in which the spine has been contracted.
\end{proof}

In the case of a symmetric linearization we can say more.

\begin{proposition}\label{prop:symcont}
For $L=(\frac{d+1}{n},\ldots,\frac{d+1}{n})$, the morphism $\varphi : \M_{0,n} \rightarrow V_{d,n}\quotient_L\text{SL}_{d+1}$ contracts $F$-curves corresponding to $n=\sum_{j=1}^4 n_j$, $n_1\le\cdots\le n_4$, satisfying either of the following conditions, and no others:
\begin{itemize}
\item $n_j \ge \frac{n\alpha_j}{d+1},~j=1,\ldots,4$, for $0\le \alpha_1 \le \cdots \le \alpha_4\in\mathbb{Z}$ with $\sum_{j=1}^4 \alpha_j = d$
\item $n_j \le \frac{n\beta_j}{d+1},~j=1,\ldots,4$, for $1\le \beta_1 \le \cdots \le \beta_4\in\mathbb{Z}$ with $\sum_{j=1}^4 \beta_j = d+2$
\end{itemize}
\end{proposition}

\begin{proof}
The contractions described by the first inequality follow from Proposition \ref{prop:cont} and are equivalent to the spine being given degree zero, since the weight on the $j^{\text{th}}$ leg in this case is $\frac{d+1}{n}n_j$.  The only other way for the F-curve to be contracted is if the spine is given positive degree but varying the cross-ratio of its attaching points no longer yields a 1-dimensional family of orbit closures.  Since we can assume each orbit closure is full-dimensional, this occurs precisely when the configuration of points in $\mathbb{P}^d$ is supported on a collection of at most $d+2$ points, and we will show that this occurs precisely when the second inequality in this proposition is satisfied.  Since the weight on the $j^{\text{th}}$ leg is $\frac{d+1}{n}n_j$, the inequality $n_j \le \frac{n\beta_j}{d+1}$ is equivalent to the $j^{\text{th}}$ leg having degree $ \le \beta_j - 1$, by an argument similar to the one in the proof of Proposition \ref{prop:cont}.  Now each leg of an F-curve is a chain of $\mathbb{P}^1s$ in which each component carries a single marked point, except for the last component which carries two marked points.  Therefore, the image of the leg having degree $\le \beta_j -1$ is equivalent to the image carrying $\le \beta_j$ distinct points.  Thus the condition $\sum_{j=1}^4\beta_j = d+2$ implies that the image of the entire curve carries $\le d+2$ distinct points.
\end{proof}

\subsection{Hassett spaces} Following \cite{Has03}, for $\mathcal{A} := \{a_1,\ldots,a_n\}\in\mathbb{Q}^n\cap (0,1]^n$ such that $\sum_{i=1}^n a_i \ge 2$ we denote by $\overline{\mathcal{M}}_{0,\mathcal{A}}$ the compact moduli space of $n$-pointed nodal rational curves for which $K + a_1p_1 + \cdots + a_np_n$ is ample and such that the $p_i$ are smooth and $\sum_{i\in I}a_i \le 1$ if $\{p_i\}_{i\in I}$ coincide.  Setting $\mathcal{A} = \{1,\ldots,1\}$ gives the usual moduli space $\M_{0,n}$.  Note that $L=(x_1,\ldots,x_n)\in\Delta(d+1,n)\cap\mathbb{Q}^n_{>0}$ is a valid choice of weight data for a Hassett space since $x_i \in \mathbb{Q}\cap(0,1]$, and $d \ge 1$ so $\sum_{i=1}^n x_i \ge 2$.  For any vector of weights $\mathcal{A}' = \{a_1',\ldots,a_n'\}$ with $a_i' \le a_i$, there is a birational contraction $\M_{0,\mathcal{A}} \rightarrow \M_{0,\mathcal{A}'}$ \cite[Theorem 4.1]{Has03}.

\begin{proposition}\label{prop:factor}
For each $L\in\Delta(d+1,n)$, the morphism $\varphi : \M_{0,n} \rightarrow V_{d,n}\quotient_L\text{SL}_{d+1}$ factors through $\M_{0,L}$.
\end{proposition}

\begin{proof}
By a result of Alexeev (see \cite[Lemma 4.6]{Fak11}) it is enough to show that every F-curve contracted by the morphism $\M_{0,n} \rightarrow \M_{0,L}$ is also contracted by $\varphi$.  This morphism contracts precisely those F-curves satisfying \begin{equation}\label{eqn:Has} \sum_{i\in N_1}x_i + \sum_{i\in N_2}x_i + \sum_{i\in N_3}x_i \le 1 \end{equation}
 where without loss of generality the leg $N_4$ carries the most weight.  Indeed, by the ampleness condition in Hassett's definition of stability a component with one node must have weight $>1$, so if (\ref{eqn:Has}) is satisfied then the first three legs must be contracted and all that remains is the spine and the fourth leg---but then the spine becomes a component with one node and $\le 1$ weight, so it too is contracted.  Conversely, if (\ref{eqn:Has}) does not hold, then either all three of the first legs are contracted or at least one of the first three legs remains.  In the first case the spine remains since it then carries $>1$ weight.  In the second case at least one of the first three legs remains so the spine remains as well, since it has two nodes and hence is Hassett-stable regardless of its weight.

The total weight of points is $d+1$, so (\ref{eqn:Has}) is equivalent to $\sum_{i\in N_4} x_i \ge d$, thus by taking $(\alpha_1,\alpha_2,\alpha_3,\alpha_4)=(0,0,0,d)$ in Proposition \ref{prop:cont} we see that all the F-curves contracted by Hassett are contracted by $\varphi$ as well. 
\end{proof}

Note that for $L\in\Delta(d+1,n)$ the Hassett space $\M_{0,L}$ parameterizes trees of $\mathbb{P}^1$s with at most $d$ leaves.  Indeed, as noted above a component with exactly one node must carry $> 1$ weight for the curve to be stable.  But there is $d+1$ weight total, so there is not enough weight to have $\ge d+1$ such components.

\subsection{Twisted cubics} We conclude this section with an example: $d=3$.  By Proposition \ref{prop:factor}, $\varphi$ factors as $\M_{0,n} \rightarrow \M_{0,L} \rightarrow V_{3,n}\quotient_L\text{SL}_4$, and the first arrow is well-understood so we focus on the second.  Here $\M_{0,L}$ parameterizes trees with at most three leaves, each leaf with weight $>1$.  Each Hassett-stable curve $C$ gets mapped to a degree 3 quasi-Veronese curve, namely a degeneration of a twisted cubic.  If $C$ is smooth then it gets mapped to a twisted cubic.  If $C$ is nodal then any component with weight $>2$ gets sent to a plane conic in $\mathbb{P}^3$, any component with weight $>1$ gets sent to a line, and all the marked points on components that are contracted lie at the corresponding singular point of the image curve in $\mathbb{P}^3$.  Figure \ref{fig:cubic} shows all the types of curves that can occur.  The number on a component indicates the degree of its image.  In all cases the image curve is non-degenerate, since if it were contained in a $\mathbb{P}^2$ then the configuration of points would be unstable.

\begin{figure}\begin{center}
\scalebox{0.8}{\input{cubic.pst}}
\caption{The morphism $\M_{0,L} \rightarrow V_{3,n}\quotient_L\text{SL}_4$}
\label{fig:cubic}
\end{center}\end{figure}

\section{Conformal Blocks}

We refer to the literature for background on conformal blocks, e.g. \cite{TUY89,BL94,Loo95,Bea96}.  The main point for us is that for each Lie algebra $\mathfrak{g}$, level $l\in\mathbb{Z}_{\ge 0}$, and $n$-tuple of dominant weights, there is a vector bundle defined on $\M_{g,n}$.  We focus on the conformal blocks bundles $D_k^{\mathfrak{sl}_n}$ defined on $\M_{0,n}$ corresponding to $\mathfrak{sl}_n$ with level $l=1$ and weights $(\omega_k,\ldots,\omega_k)$.  These are line bundles (by \cite[\S5.2.2]{Fak11}, or by the Verlinde formula) and they are the focus of the paper \cite{AGSS10}.

\subsection{Fakhruddin's formula} In \cite{Fak11}, Fakhruddin computes the Chern classes of the conformal blocks vector bundles on $\M_{0,n}$.  This is then used to find the intersection numbers between their determinants and F-curves.  The formula for the line bundles $D_k^{\mathfrak{sl}_n}$ is as follows \cite[Proposition 5.2]{Fak11}:

\begin{proposition} Let $F$ be an F-curve corresponding to $n = n_1 + \cdots + n_4$.  Let $\nu_j = kn_j~(\text{mod}~n)\in \{0,1,\ldots,n-1\}$ and $\nu_M = \text{max}\{\nu_j\}, \nu_m = \text{min}\{\nu_j\}$.  Then \[D_k^{\mathfrak{sl}_n}\cdot F = \begin{cases} \nu_m$, if $\sum_{j=1}^4 \nu_j = 2n$ and $\nu_M + \nu_m \le n \cr n - \nu_M$, if $\sum_{j=1}^4 \nu_j = 2n$ and $\nu_M + \nu_m \ge n \cr 0, $ otherwise $ \end{cases}\]
\end{proposition}

From this case of Fakhruddin's formula we deduce the following:

\begin{proposition}\label{prop:cont2}
The complete linear system $|D_k^{\mathfrak{sl}_n}|$ contracts F-curves $n=\sum_{j=1}^4 n_j$, $n_1\le\cdots\le n_4$, satisfying either of the following conditions:
\begin{itemize}
\item $n_j \ge \frac{n\alpha_j}{k},~j=1,\ldots,4$, for $0\le \alpha_1 \le \cdots \le \alpha_4\in\mathbb{Z}$ with $\sum_{j=1}^4 \alpha_j = k-1$
\item $n_j \le \frac{n\beta_j}{k},~j=1,\ldots,4$, for $1\le \beta_1 \le \cdots \le \beta_4\in\mathbb{Z}$ with $\sum_{j=1}^4 \beta_j = k+1$
\end{itemize}
\end{proposition}

\begin{proof} 
By Fakhruddin's formula, $D_k^{\mathfrak{sl}_n}\cdot F = 0$ if and only if $\sum_{j=1}^4 \nu_j \ne 2n$ or $\nu_j = 0$ for some $j$, where $\nu_j = kn_j~(\text{mod}~n)$.  If there are $\alpha_j\ge 0$ with $kn_j \ge n\alpha_j$, then $\nu_j \le kn_j - n\alpha_j$.  Together with the fact that $\sum_{j=1}^4 \alpha_j = k-1$, we then have \[\sum_{j=1}^4 \nu_j \le \sum_{j=1}^4 kn_j - \sum_{j=1}^4 n\alpha_j = kn - n(k-1) = n \ne 2n\] so $D_k^{\mathfrak{sl}_n}\cdot F = 0$.  Similarly, if there are $\beta_j \ge 1$ with $kn_j \le n\beta_j$, then either $\nu_j = 0$ for some $j$ or $\nu_j \ge kn_j - n(\beta_j-1)$.  Then $\sum_{j=1}^4 \beta_j = k+1$ implies that \[\sum_{j=1}^4 \nu_j \ge \sum_{j=1}^4 kn_j - \sum_{j=1}^4 n(\beta_j-1) = kn - n(k+1) + 4n = 3n \ne 2n\] so again $D_k^{\mathfrak{sl}_n}\cdot F = 0$.
\end{proof}

For our present purposes we only need to know that $|D_k^{\mathfrak{sl}_n}|$ contracts the F-curves described by this proposition, but in fact we will see later (Corollary \ref{cor:cont}) that no other F-curves are contracted, so this proposition precisely describes the zeroes of the Fakhruddin formula in the case of $D_k^{\mathfrak{sl}_n}$.

\subsection{Proof of Theorem \ref{thm:CB}} Recall that for symmetric linearization $L$ we denote the polarization on $V_{d,n}\quotient_L\text{SL}_{d+1}$ by $\mathcal{L}_d$.  The line bundle $\varphi^*\mathcal{L}_d$ is nef and $S_n$-invariant, so it spans a ray in $\text{SymNef}(\M_{0,n})\subseteq N^1(\M_{0,n})^{S_n}$.  This latter vector space has dimension $\rho := \lfloor\frac{n}{2}\rfloor - 1$ \cite[\S2.2.1]{AGSS10}.  Each linearly independent curve contracted by $\varphi$ forces $\varphi^*\mathcal{L}_d$ to lie on a face of the symmetric nef cone of one higher codimension.  Thus if we find $\rho - 1$ independent curves that are contracted, then $\varphi^*\mathcal{L}_d$ will span an extremal ray of $\text{SymNef}(\M_{0,n})$ and this set of curves uniquely determines the ray.  

The conformal blocks line bundle $D^{\mathfrak{sl}_n}_{d+1}$ is also nef and $S_n$-invariant, and in \cite{AGSS10} a set of $\rho - 1$ independent symmetric F-curves that it contracts is described.  Thus to show that $\varphi^*\mathcal{L}_d$ and $D^{\mathfrak{sl}_n}_{d+1}$ span the same ray it is enough to show that $\varphi$ contracts this same set of F-curves.  The set used in \cite{AGSS10} depends on certain divisibility relations between $d+1$ and $n$, but in each case one can verify that the curves it contains satisfy the hypotheses of Proposition \ref{prop:cont2}, so that by Proposition \ref{prop:symcont} they are contracted by $\varphi$.

For example, when $(d+1)|n$ write $q := \frac{n}{d+1}$ and let $F_{a,b,c}$ denote a symmetric F-curve corresponding to $n=a+b+c+(n-(a+b+c))$.  The family is then defined as \[\{F_{i,1,1}~|~1\le i \le \lfloor\frac{n}{2}\rfloor-1, q \nmid (i+1)\}\cup \{F_{i,q,q}~|~1\le i \le \lfloor\frac{n}{2}\rfloor-1, q|(i+1)\}\] with the last $F_{i,q,q}$ curve removed \cite[\S6.2]{AGSS10}.  We claim the $F_{i,1,1}$ are contracted by the first inequality in Proposition \ref{prop:symcont}, whereas the $F_{i,q,q}$ are contracted by the second inequality.  For this range of $i$, we have $i\le n-i-2$, so $F_{i,1,1}$ corresponds to $n=\sum_{j=1}^4 n_j$ with $(n_1,n_2,n_3,n_4) = (1,1,i,n-i-2)$, $n_1 \le \cdots \le n_4$.  Setting $(\alpha_1,\alpha_2,\alpha_3,\alpha_4) = (0,0,\lfloor\frac{i}{q}\rfloor,d-\lfloor\frac{i}{q}\rfloor)$, we only have to check that $n_j \ge q\alpha_j$ for $j=3,4$.  For $j=3$ this is obvious: $n_3 = i \ge q\lfloor\frac{i}{q}\rfloor$.  For $j=4$ we have to verify the inequality $n-i-2 \ge q(d-\lfloor\frac{i}{q}\rfloor)$, or equivalently, using that $n=q(d+1)$, the inequality $q\lfloor\frac{i}{q}\rfloor \ge i + 2 - q$.  In general, $\lfloor\frac{i}{q}\rfloor \ge \frac{i-q+1}{q}$ holds, but the hypothesis $q\nmid(i+1)$ implies the stronger inequality $\lfloor\frac{i}{q}\rfloor \ge \frac{i-q+2}{q}$, from which the desired inequality immediately follows.  For the $F_{i,q,q}$ curves we have $(n_1,n_2,n_3,n_4)=(i,n-i-2q,q,q)$, so we set $(\beta_1,\beta_2,\beta_3,\beta_4)=(\lceil\frac{i}{q}\rceil,d-\lceil\frac{i}{q}\rceil,1,1)$ and must check $n_j \le q\beta_j$ for $j=1,2,3,4$.  This is automatic for $j=1,3,4$, and for $j=2$ we must verify $n-i-2q\le q(d-\lceil\frac{i}{q}\rceil)$.  This latter inequality boils down to $\lceil\frac{i}{q}\rceil \le \frac{i}{q} + 1$, which holds even without the hypothesis that $q|(i+1)$. 

Thus the family of $\rho-1$ curves is contracted by $\varphi$ when $(d+1)|n$.  The remaining cases \cite[Definitions 7.1 and 9.1]{AGSS10} are similar. \hfill $\Box$

\begin{remark}
The families of contracted F-curves used to prove extremality in \cite{AGSS10} were found experimentally and do not appear to have a very geometric description.  The preceding proof sheds some light on their definition, since it gives a geometric reason why these particular F-curves are contracted.  In the case $(d+1)|n$ we witness some structure to the family: it consists of two types of curves, $F_{i,1,1}$ and $F_{i,q,q}$, and our proof shows that these correspond to the two types of geometric contractions described in Proposition \ref{prop:symcont}.  When $(d+1)\nmid n$ the family still consists of curves that fall into these two types of geometric contractions, but the dichotomy is less transparent.  Moreover, the heart of the extremality proof in \cite{AGSS10} is to prove that the curves in each family are independent, and our present approach does not seem to reveal anything in this regard.  It would be interesting to find a purely geometric construction of the families in all cases.
\end{remark}

\subsection{Consequences} We immediately derive two results from Theorem \ref{thm:CB} and the work leading up to it.  One is that since $D_{d+1}^{\mathfrak{sl}_n}$ and $\varphi^*\mathcal{L}_d$ span the same ray in $N^1(\M_{0,n})$, and the former spans an extremal ray in the symmetric nef cone, the latter must do so as well.  This is Corollary \ref{cor:extremal}.

A second consequence is that $D_{d+1}^{\mathfrak{sl}_n}$ and $\varphi^*\mathcal{L}_d$ contract the same set of F-curves, but Proposition \ref{prop:symcont} completely describes that set of curves, so we see that the list in Proposition \ref{prop:cont2} must be complete as well:

\begin{corollary}\label{cor:cont}
The complete linear system $|D_k^{\mathfrak{sl}_n}|$ does not contract any further F-curves than those listed in Proposition \ref{prop:cont2}.
\end{corollary}

A third consequence is described in the following section.

\section{Gale Duality}\label{section:Gale}

\subsection{Background} The Gale transform, sometimes called \emph{association}, is a classical construction associating to a configuration of $n$ sufficiently general points in $\mathbb{P}^d$, up to projectivity, a configuration of $n$ general points in $\mathbb{P}^{n-d-2}$, up to projectivity.    In \cite[Corollary III.1]{DO88} this was extended, using GIT, to partial compactifications of the space of general configurations, yielding \[(\mathbb{P}^d)^n\quotient_L\text{SL}_{d+1} \cong (\mathbb{P}^{n-d-2})^n\quotient_L\text{SL}_{n-d-1}\] where $L$ denotes a symmetric linearization.  Note that the space of linearizations on the left-hand side is the hypersimplex $\Delta(d+1,n)$, whereas on the right-hand side it is the hypersimplex $\Delta(n-d-1,n)$.  However, these two polytopes are naturally isomorphic: the former is the convex hull of all $\binom{n}{d+1}$ vectors in $\{0,1\}^n$ with 1 occurring in $d+1$ entries and the latter is the convex hull of those with 0 occurring in $d+1$ entries.  Under this identification, it was shown in \cite{Hu05b} that the above isomorphism of GIT quotients applies for arbitrary linearizations, not just symmetric ones.  The Gale transform extends to Chow quotients as well \cite[Corollary 2.3.14]{Kap93}.

Perhaps the easiest way to describe the Gale transform is in terms of Grassmannians, via the Gelfand-MacPherson correspondence \cite{GM82}.  For both Chow and GIT quotients there is an isomorphism \[(\mathbb{P}^d)^n\quotient\text{SL}_{d+1} \cong \text{Gr}(d+1,n)\quotient (\mathbb{C}^*)^n\] coming from the fact that both sides may be viewed as double quotients of the same space \cite[\S2.2]{Kap93}.  Indeed, $(\mathbb{C}^*)^n$ acts on the space of matrices $M_{d+1\times n}$ by scaling each column, and $\text{GL}_{d+1}$ acts on the left by multiplication.  The quotient by the torus first then $\text{GL}_{d+1}$ is the left-hand side of this isomorphism, and the quotient by $\text{GL}_{d+1}$ then the torus is the right-hand side.  The Gale transform may then be viewed as an isomorphism \[\text{Gr}(d+1,n)\quotient(\mathbb{C}^*)^n \cong \text{Gr}(n-d-1,n)\quotient(\mathbb{C}^*)^n\] which is obtained by simply taking orthogonal complements.  One sees immediately from this description that the Gale transform is involutive.

\subsection{Configurations on curves}\label{section:configs} As is described in \cite{EP00}, Goppa recognized that the Gale transform of a configuration of $n$ distinct points supported on a rational normal curve in $\mathbb{P}^d$ is supported on a rational normal curve in $\mathbb{P}^{n-d-2}$.  Thus the Gale transform $\Gamma$ sends $U_{d,n}/\text{SL}_{d+1}$ to a subset of $U_{n-d-2,n}/\text{SL}_{n-d-1}$, and similarly $\Gamma(U_{n-d-2,n}/\text{SL}_{n-d-1})\subseteq U_{d,n}/\text{SL}_{d+1}$.  But the Gale transform is an involutive isomorphism, so the arrows in the top row of the following commutative diagram must be isomorphisms:
\[\xymatrix{ U_{d,n}/\text{SL}_{d+1} \ar@{^{(}->}[d] \ar[r] & U_{n-d-2,n}/\text{SL}_{n-d-1} \ar[r] \ar@{^{(}->}[d] & U_{d,n}/\text{SL}_{d+1} \ar@{^{(}->}[d] \\ V_{d,n}\quotient_L\text{SL}_{d+1} \ar@{^{(}->}[d] \ar[r] & V_{n-d-2,n}\quotient_L\text{SL}_{n-d-1} \ar[r] \ar@{^{(}->}[d] & V_{d,n}\quotient_L\text{SL}_{d+1} \ar@{^{(}->}[d] \\ (\mathbb{P}^d)^n\quotient_L\text{SL}_{d+1} \ar[r] & (\mathbb{P}^{n-d-2})^n\quotient_L\text{SL}_{n-d-1} \ar[r] & (\mathbb{P}^d)^n\quotient_L\text{SL}_{d+1}}\]
By Lemma \ref{lem:Vdef} we can view $V_{d,n}$ as the topological closure of $U_{n,d}$ so the Gale transform restricts to give maps as in the middle row of this diagram, and for the same reason they must be isomorphisms.  Therefore, Goppa's result extends to quasi-Veronese curves: the Gale transform of a configuration supported on a quasi-Veronese curve lies on a quasi-Veronese curve.  Note that in \cite[Corollary 3.2]{EP00} it is shown that Goppa's result extends to Gorenstein curves, but our quasi-Veronese curves include examples of non-Gorenstein curves.  For instance, a spatial triple point as in Figure \ref{fig:cubic} is not Gorenstein.

\subsection{Conformal blocks symmetry} It is well-known that $D^{\mathfrak{sl}_n}_{k} = D^{\mathfrak{sl}_n}_{n-k}$.  One way to see this is by the symmetry of the Dynkin diagram $\circ-\circ-\cdots-\circ-\circ$ for $\mathfrak{sl}_n$.  From Theorem \ref{thm:CB} it then follows that if 
\[\varphi_1 : \M_{0,n} \rightarrow V_{d,n}\quotient_L\text{SL}_{d+1},~\varphi_2 : \M_{0,n} \rightarrow V_{n-d-2,n}\quotient_L\text{SL}_{n-d-1}\] 
then the pulled back symmetric polarizations $\varphi_1^*\mathcal{L}_d$ and $\varphi_2^*\mathcal{L}_{n-d-2}$ span the same ray in $N^1(\M_{0,n})$.  Thus $\varphi_1$ and $\varphi_2$ have the same Stein factorizations up to a finite map, so the normalizations of their images are isomorphic: \[(V_{d,n}\quotient_L\text{SL}_{d+1})^\nu\cong (V_{n-d-2,n}\quotient_L\text{SL}_{n-d-1})^\nu\]  This proves Corollary \ref{cor:Gale}.  Of course, this isomorphism also follows from \S\ref{section:configs}, but by allowing the normalization here we derive the result without, strictly speaking, invoking the Gale transform.

\subsection{Self-association of $\M_{0,2m}$}\label{section:selfassoc}
Consider configurations of $2m$ points in $\mathbb{P}^{m-1}$.  Here the source and target are the same, so the Gale transform provides an involution $\Gamma\in\text{Aut}((\mathbb{P}^{m-1})^{2m}\ChowQ\text{SL}_m)$.  Fixed points of this automorphism are called \emph{self-associated} configurations.  Such configurations have been studied classically.  See \cite{DO88,EP00} and the references therein, as well as \cite[Corollary 2.3.10]{Kap93}.  It is proven in \cite{Fla98} that any configuration of $2m$ distinct points on a rational normal curve of degree $m-1$ is self-associated, so $\Gamma$ restricts to the identity on $U_{m-1,2m}/\text{SL}_m \subseteq (\mathbb{P}^{m-1})^{2m}\ChowQ\text{SL}_m$.  It therefore is the identity on the topological closure of this locus as well.  But the closure is precisely the Chow quotient $V_{m-1,2m}\ChowQ\text{SL}_m$, which by Theorem \ref{thm:morphism} is isomorphic to $\M_{0,2m}$.  Thus, under the identification \[\M_{0,2m}~\widetilde{\rightarrow}~V_{m-1,2m}\ChowQ\text{SL}_m \hookrightarrow (\mathbb{P}^{m-1})^{2m}\ChowQ\text{SL}_m\] each point $(C,p_1,\ldots,p_{2m})$ corresponds to a self-associated configuration.

\bibliographystyle{amsalpha}

\end{document}